\documentclass[12pt]{article}

\usepackage{cite}
\usepackage{amsfonts, amssymb, amsmath, amsthm, epsf}

\numberwithin{figure}{section}

\theoremstyle{plain}
\newtheorem{theorem}{Theorem}[section]
\newtheorem{lemma}{Lemma}[section]

\newtheorem{corollary}{Corollary}[section]
\newtheorem{condition}{Condition}[section]

\theoremstyle{definition}

\newtheorem{remark}{Remark}[section]

\numberwithin{equation}{section}

\renewcommand{\Im}{{\rm Im\,}}

\newcommand{\loc}{{\rm loc}}
\newcommand{\ind}{{\rm ind\,}}
\newcommand{\supp}{{\rm supp\,}}

\newcommand{\Dom}{{\rm D}}

\renewcommand{\ker}{{\rm ker\,}}
\renewcommand{\dim}{{\rm dim\,}}

\newcommand{\const}{{\rm const}}
\newcommand{\dist}{{\rm dist}}
\renewcommand{\phi}{{\varphi}}
\newcommand{\n}{{|\!|\!|}}

\newcommand{\cK}{{\mathcal K}}
\newcommand{\cM}{{\mathcal M}}

\newcommand{\cL}{{\mathcal L}}
\newcommand{\cB}{{\mathcal B}}
\newcommand{\cG}{{\mathcal G}}
\newcommand{\cO}{{\mathcal O}}
\newcommand{\cC}{{\mathcal C}}

\newcommand{\cH}{{\mathcal H}}

\newcommand{\cX}{{\mathcal X}}

\newcommand{\pG}{{\partial G}}

\newcommand{\oG}{{\overline G}}

\newcommand{\bP}{{\mathbf P}}
\newcommand{\bB}{{\mathbf B}}
\newcommand{\bL}{{\mathbf L}}

\newcommand{\bI}{{\mathbf I}}
\newcommand{\bT}{{\mathbf T}}

\newcommand{\bbR}{{\mathbb R}}

\title{On the existence of a Feller semigroup with atomic measure in nonlocal boundary condition}
\author{Pavel Gurevich}

\date{}

\begin{document}

\maketitle


\begin{abstract}
The existence of Feller semigroups arising in the theory of multidimensional diffusion
processes is studied. An elliptic operator of second order is considered on a plane
bounded region~$G$. Its domain of definition consists of continuous functions satisfying
a nonlocal condition  on the boundary of the region. In general, the nonlocal term is an
integral of a function over the closure of the region $G$ with respect to a nonnegative
Borel measure $\mu(y,d\eta)$, $y\in\pG$. It is proved that the operator is a generator of
a Feller semigroup in the case where the measure is atomic. The smallness of the measure
is not assumed.
\end{abstract}

\section{Introduction and Preliminaries}

It was shown in~\cite{Feller1,Feller2} that any one-dimensional diffusion process is
related to a strongly continuous contractive nonnegative semigroup  (the {\it Feller
semigroup\/})  of operators acting on the space of continuous functions. Moreover, a
general form of the generator of this semigroup was obtain and all possible boundary
conditions defining its domain were described.

In the multidimensional case, a general form of the generator of a Feller semigroup was
obtained in~\cite{Ventsel}. It was proved that the generator of a Feller semigroup is an
elliptic differential operator of second order (perhaps, degenerated) whose domain of
definition consists of continuous functions satisfying a nonlocal boundary condition. The
nonlocal term is given by the integral of a function over the closure of a region $G$
with respect to a nonnegative Borel measure  $\mu(y,d\eta)$, $y\in\pG$.

The following problem is unsolved. Given an elliptic differential operator of second
order whose domain is defined by a general nonlocal condition (see~\cite{Ventsel}),
whether or not its closure is a generator of a Feller semigroup?

One  distinguishes the {\it transversal} and {\it nontransversal} nonlocal conditions. In
the transversal case, the order of nonlocal terms is less than the order of the local
terms, whereas these orders coincide in the nontransversal case. The transversal case was
considered in~\cite{SU,BCP,Watanabe,Taira1,Taira3,Ishikawa}. A method of the study of the
more difficult nontransversal case was developed in
papers~\cite{SkDAN89,SkRJMP95,GalSkMs,GalSkJDE}. These  works are devoted to obtaining
sufficient conditions on the coefficients and the Borel measure (in the nonlocal
condition) that ensure the existence of a Feller semigroup.

In~\cite{GalSkMs,GalSkJDE}, the authors study  the case where the measure  $\mu(y,\oG)$
(after some normalization) is less than one. In this paper, we investigate  {\it
nontransversal\/} nonlocal conditions given on the boundary of a plane bounded domain
$G$, admitting the ``limit case'' in which the measure $\mu(y,\oG)$ may equal one (it
cannot be greater than one~\cite{Ventsel}). We consider a model case where the measure
$\mu(y,d\eta)$ is atomic and vanishes for $y$ from outside of some
$\varepsilon$-neighborhood of a set $\cK\subset\pG$ consisting of finitely many points.

By using theorems on the solvability of elliptic equations with nonlocal boundary
conditions in the Kondrat'ev weighted spaces~\cite{GurSkub}, asymptotics of solutions
near the conjugation points~\cite{GurPetr03} (the points of the set $\cK$), and the
maximum principle, we investigate the solvability of nonlocal problems in the spaces of
continuous functions (see
Secs.~\ref{subsectStatement}--\ref{subsectNonlocalProblemsInC}). Applying these results
and the Hille--Iosida theorem, we prove in Sec.~\ref{secFeller} that an elliptic operator
with the above nonlocal boundary conditions is a generator of a Feller semigroup.

In the conclusion of this section, we remind the notion of a Feller semigroup and its
generator and formulate a version of the Hille--Iosida theorem adapted for our purposes.

\bigskip

Let $G\subset\bbR^2$ be a bounded domain with piecewise smooth boundary $\pG$, and let
$X$ be a closed subspace in $C(\oG)$ containing at least one nontrivial nonnegative
function.

A strongly continuous semigroup of operators $\bT_t:X\to X$ is called a {\it Feller
semigroup  on} $X$ if it satisfies the following conditions: 1. $\|\bT_t\|\le 1$,
$t\ge0$; 2.  $\bT_t u\ge0$ for all $t\ge0$ and $u\in X$, $u\ge0$.

A linear operator $\bP:\Dom(\bP)\subset X\to X$ is called the ({\it infinitesimal\/})
{\it generator} of a strongly continuous semigroup $\{\bT_t\}$ if $$ \bP
u=\lim\limits_{t\to +0}{(\bT u-u)}/{t},\quad  \Dom(\bP)=\{u\in X:  \text{the limit exists
in } X\}. $$

\begin{theorem}[the Hille--Iosida theorem, see Theorem~9.3.1 in~\cite{Taira1}]\label{thHI}
\begin{enumerate}
\item
Let  $\bP:\Dom(\bP)\subset X\to X$  be a generator of a Feller semigroup on $X$. Then the
following assertions are true.
\begin{enumerate}
\item[$(a)$]
The domain $\Dom(\bP)$ is dense in $X$.
\item[$(b)$]
For each $q>0$ the operator $q\bI-\bP$ has the bounded inverse $(q\bI-\bP)^{-1}:X\to X$
and $\|(q\bI-\bP)^{-1}\|\le 1/q$.
\item[$(c)$]
The operator $(q\bI-\bP)^{-1}:X\to X$, $q>0$, is nonnegative.
\end{enumerate}
\item
Conversely, if $\bP$ is a linear operator from $X$ to $X$ satisfying condition $(a)$ and
there is a constant $q_1\ge 0$ such that conditions $(b)$ and $(c)$ hold for $q>q_1$,
then $\bP$ is the generator of a certain Feller semigroup  on $X$, which is uniquely
determined by $\bP$.
\end{enumerate}
\end{theorem}

\section{Setting of Nonlocal Problems}\label{subsectStatement}

Let $G\subset{\mathbb R}^2$ be a bounded domain with boundary
$\partial G$. Consider a set ${\cK}\subset\partial G$ consisting
of finitely many points. Let $\partial G\setminus{\mathcal
K}=\bigcup\limits_{i=1}^{N}\Gamma_i$, where $\Gamma_i$ are open
(in the topology of $\partial G$) $C^\infty$ curves. Assume that
the domain $G$ is a plane angle in some neighborhood of each point
$g\in{\mathcal K}$.

For an integer $k\ge0$, denote by $W_2^k(G)$ the Sobolev space. By $W^k_{2,\loc}(G)$ we
denote the set of functions  $u$ such that  $u\in W_2^k(G')$ for any $G'$,
$\overline{G'}\subset G$.

If $\cX$ is a domain in $\mathbb R^2$, we denote by $C_0^\infty(\cX)$ the set of
functions infinitely differentiable on $\overline{ \cX}$ and compactly supported on
$\cX$. If $\cM\subset\cX$, we denote by $C_0^\infty(\overline \cX\setminus \cM)$ the set
of functions infinitely differentiable on $\overline{ \cX}$ and compactly supported on
$\overline \cX\setminus \cM$.

Along with Sobolev spaces, we will use weighted spaces (the Kondrat'ev spaces). Let
$Q=\{y\in{\mathbb R}^2:  r>0,\ |\omega|<\omega_0\}$, $Q=\{y\in{\mathbb R}^2:  0<r<d,\
|\omega|<\omega_0\}$, $0<\omega_0<\pi$, $d>0$, or $Q=G$. We denote by $\mathcal M$ the
set $\{0\}$ in the first and second cases and the set $\mathcal K$ in the third case.
Introduce the space $H_a^k(Q)=H_a^k(Q,\mathcal M)$ as the completion of the set
$C_0^\infty(\overline{ Q}\setminus \mathcal M)$ with respect to the norm
$$
 \|u\|_{H_a^k(Q)}=\Bigg(
    \sum_{|\alpha|\le k}\,\int\limits_Q \rho^{2(a+|\alpha|-k)} |D^\alpha u(y)|^2 dy
                                       \Bigg)^{1/2},
$$
where $a\in \mathbb R$, $k\ge 0$ is an integer, and
$\rho=\rho(y)=\dist(y,\mathcal M)$. For an integer $k\ge1$, denote
by $H_a^{k-1/2}(\Gamma)=H_a^{k-1/2}(\Gamma,\cM)$  the set of
traces on a smooth curve $\Gamma\subset\overline{ Q}$ (with the
 infimum-norm).

Let $p_{jk},p_j\in C^\infty(\bbR^2)$ be real-valued functions, and
let $p_{jk}=p_{kj}$, $j,k=1,2$. Consider the differential operator
\begin{equation}\label{eq1}
Pu=\sum\limits_{j,k=1}^{2}p_{jk}(y)u_{y_jy_k}(y)+
\sum\limits_{j=1}^2p_j(y)u_{y_j}(y)+p_0(y)u(y).
\end{equation}

\begin{condition}\label{cond1.1}
1. There is a constant $c_0>0$ such that $\sum\limits_{j,k=1}^{2}p_{jk}(y)\xi_j\xi_k\ge
c_0|\xi|^2$ for $y\in\overline{G}$ and $\xi=(\xi_1,\xi_2)\in\bbR^2.$

2. $p_0(y)\le0$ for $y\in\overline{G}$.
\end{condition}

Introduce the operators corresponding to nonlocal terms supported near the set $\mathcal
K$. For any  set $\mathcal M$, we denote its $\varepsilon$-neighborhood by $\mathcal
O_{\varepsilon}(\mathcal M)$. Let $\Omega_{is}$ ($i=1, \dots, N;$ $s=1, \dots, S_i$) be
$C^\infty$-diffeomorphisms taking some neighborhood ${\mathcal O}_i$ of the curve
$\overline{\Gamma_i\cap\mathcal O_{{\varepsilon}}(\mathcal K)}$ to the set
$\Omega_{is}({\mathcal O}_i)$ in such a way that $\Omega_{is}(\Gamma_i\cap\mathcal
O_{{\varepsilon}}(\mathcal K))\subset G$ and $ \Omega_{is}(g)\in\mathcal K\ \text{for}\
g\in\overline{\Gamma_i}\cap\mathcal K. $ Thus, the transformations $\Omega_{is}$ take the
curves $\Gamma_i\cap\mathcal O_{{\varepsilon}}(\mathcal K)$   inside the domain $G$ and
the set of their end points $\overline{\Gamma_i}\cap\mathcal K$ to itself.

 Denote by $\Omega_{is}^{+1}$ the
transformation $\Omega_{is}:{\mathcal O}_i\to\Omega_{is}({\mathcal
O}_i)$ and by $\Omega_{is}^{-1}:\Omega_{is}({\mathcal
O}_i)\to{\mathcal O}_i$ the inverse transformation. The set of
points
$\Omega_{i_qs_q}^{\pm1}(\dots\Omega_{i_1s_1}^{\pm1}(g))\in{\mathcal
K}$ ($1\le s_j\le S_{i_j},\ j=1, \dots, q$) is said to be an {\em
orbit} of the point $g\in{\mathcal K}$. In other words, the orbit
of a point $g\in\cK$ is formed by the points (of the set $\mathcal
K$) that can be obtained by consecutively applying the
transformations $\Omega_{i_js_j}^{\pm1}$ to the point~$g$.

The set $\mathcal K$ consists of finitely many disjoint orbits,
which we denote by $\mathcal K_\nu$, $\nu=1,\dots,N_0$. Fix an
arbitrary orbit $\mathcal K_\nu$ and assume that it consists of
points\footnote{The points $g_j$ and other objects (see
below) related to the orbit $\mathcal K_\nu$ depend on $\nu$. To
avoid clumsy notation, we do not explicitly indicate this
dependence.} $g_j$, $j=1, \dots,
N_\nu$.

Take a sufficiently small number $\varepsilon>0$ such that there
exist neighborhoods $\mathcal O_{\varepsilon_1}(g_j)$, $ \mathcal
O_{\varepsilon_1}(g_j)\supset\mathcal O_{\varepsilon}(g_j) $,
satisfying the following conditions: (1) the domain $G$ is a plane
angle in the neighborhood $\mathcal O_{\varepsilon_1}(g_j)$;   (2)
$\overline{\mathcal O_{\varepsilon_1}(g)}\cap\overline{\mathcal
O_{\varepsilon_1}(h)}=\varnothing$ for any $g,h\in\mathcal K$,
$g\ne h$; (3) if $g_j\in\overline{\Gamma_i}$ and
$\Omega_{is}(g_j)=g_k,$ then ${\mathcal
O}_{\varepsilon}(g_j)\subset\mathcal
 O_i$ and
 $\Omega_{is}\big({\mathcal
O}_{\varepsilon}(g_j)\big)\subset{\mathcal
O}_{\varepsilon_1}(g_k).$

For each point $g_j\in\overline{\Gamma_i}\cap\mathcal K_\nu$, we
fix a linear transformation $Y_j: y\mapsto y'(g_j)$ (the
composition of the shift by the vector $-\overrightarrow{Og_j}$
and rotation) mapping the point $g_j$ to the origin in such a way
that $ Y_j({\mathcal O}_{\varepsilon_1}(g_j))={\mathcal
O}_{\varepsilon_1}(0),\ Y_j(G\cap{\mathcal
O}_{\varepsilon_1}(g_j))=K_j\cap{\mathcal O}_{\varepsilon_1}(0),\
Y_j(\Gamma_i\cap{\mathcal
O}_{\varepsilon_1}(g_j))=\gamma_{j\sigma}\cap{\mathcal
O}_{\varepsilon_1}(0)\ (\sigma=1\ \text{or}\ 2), $ where
\begin{equation}\label{eqK_jgamma_jsigma}
 K_j=\{y\in{\mathbb R}^2:\ r>0,\ |\omega|<\omega_j\},\qquad \gamma_{j\sigma}=\{y\in\mathbb
R^2:\ r>0,\ \omega=(-1)^\sigma \omega_j\},
\end{equation}
 $\omega,r $ are the polar coordinates,
$0<\omega_j<\pi$. Without loss of generality, we assume that the
principal homogeneous part of the operator $P$ at the point $g_j$
is the Laplace operator in the new variables~$y'$.

\begin{condition}\label{condK1}
Let $g_j\in\overline{\Gamma_i}\cap\mathcal K_\nu$ and
$\Omega_{is}(g_j)=g_k\in\mathcal K_\nu;$ then the transformation $
Y_k\circ\Omega_{is}\circ Y_j^{-1}:{\mathcal
O}_{\varepsilon}(0)\to{\mathcal O}_{\varepsilon_1}(0) $ is the
composition of rotation and homothety.
\end{condition}

Introduce the nonlocal operators $
 \mathbf B_{i}u=\sum\limits_{s=1}^{S_i}
b_{is}(y) u(\Omega_{is}(y))$ for $
   y\in\Gamma_i\cap\mathcal O_{\varepsilon}(\mathcal K)$ and $
\mathbf B_{i}u=0$ for $y\in\Gamma_i\setminus\mathcal
O_{\varepsilon}(\mathcal K), $ where $b_{is}\in C^\infty(\mathbb
R^2)$ are real-valued functions,  $\supp b_{is}\subset\mathcal
O_{{\varepsilon}}(\mathcal K)$.

\begin{condition}\label{cond1.2}
The following relations hold{\rm :}
\begin{equation}\label{eq4-5}
b_{is}(y)\ge0,\qquad \sum\limits_{s=1}^{S_i} b_{is}(y)\le 1,\qquad
 y\in\overline{\Gamma_i};
\end{equation}
\begin{equation}\label{eq6}
\sum\limits_{s=1}^{S_i} b_{is}(g)+\sum\limits_{s=1}^{S_j}
b_{js}(g)<2,\quad
g\in\overline{\Gamma_i}\cap\overline{\Gamma_j}\subset\cK,\qquad\text{if}\
 i\ne j\ \text{and}\
 \overline{\Gamma_i}\cap\overline{\Gamma_j}\ne\varnothing.
\end{equation}
\end{condition}

We will study the nonlocal elliptic problem
\begin{equation}\label{eq2-3}
Pu-q u=f(y), \ y\in G;\qquad u|_{\Gamma_i}-\bB_i u=0,\
y\in\Gamma_i,\ i=1,\dots,N,
\end{equation}
where $q\ge0$, and the same problem with nonhomogeneous nonlocal
conditions. To consider problem~\eqref{eq2-3} in spaces of
continuous functions, we preliminarily study it in the weighted
spaces.

In the sequel, we need norms in weighted spaces depending on the
parameter $q>0$. Set
$$
\n u \n_{H_a^k(G)}=
\left(\|u\|_{H_{a}^k(G)}^2+q^k\|u\|^2_{H_a^0(G)}\right)^{1/2},\qquad
k\ge0,
$$
$$
\n v \n_{H_a^{k-1/2}(\Gamma_i)}=
\left(\|v\|_{H_a^{k-1/2}(\Gamma_i)}^2+q^{k-1/2}\|
v\|_{H_{a}^0(\Gamma_i)}^2\right)^{1/2},\qquad k\ge1,
$$
where $
\|v\|_{H_{a}^0(\Gamma_i)}=\Big(\,\int\limits_{\Gamma_i}\rho^{2a}|v(y)|^2d\Gamma\Big)^2.
$ We also consider the following spaces:
\begin{enumerate}
\item[$\bullet$]
$\cH_a^{k+3/2}(\pG)=\prod\limits_{i=1}^N H_a^{k+3/2}(\Gamma_i)$, $
\n \psi \n_{\cH_a^{k+3/2}(\pG)}=\left(\sum\limits_{i=1}^N \n
\psi_i\n_{H_a^{k+3/2}(\Gamma_i)}^2\right)^{1/2},\ \psi=\{\psi_i\},
$
\item[$\bullet$] $\cH_a^k(G,\pG)=H_a^k(G)\times \cH_a^{k+3/2}(\pG)$, $ \n (f,\psi) \n_{\cH_a^k(G,\pG)}=\left(\n f
\n_{H_a^k(G)}^2+\n\psi\n_{\cH_a^{k+3/2}(\pG)}^2 \right)^{1/2}. $
\end{enumerate}

Consider  the bounded operator $$
\begin{gathered}
 \bL(q):H_{k+1-\delta}^{k+2}(G)\to\mathcal H_{k+1-\delta}^k(G,\pG),\\
  \bL(q) u=\{Pu-q u,\ u|_{\Gamma_i}-\bB_i u\},\quad q\ge0.
\end{gathered}
  $$
   We prove the following theorem in
Sec.~\ref{secSolvabilityWeighted}.
\begin{theorem}\label{th1}
Let Conditions~$\ref{cond1.1}$--$\ref{cond1.2}$ hold, and let $k\ge0$ be fixed. Then
there exists a number $\delta_1>0$ possessing the following property{\rm :} for any
$\delta\in[0,\delta_1]$, there is a number $q_1=q_1(\delta)>0$ such that the operator
$\bL(q)$ has a bounded inverse for $q\ge q_1$ and
\begin{equation}\label{eqth1}
c \n \bL(q)u\n_{\mathcal H_{k+1-\delta}^k(G,\pG)}\le\n u\n_{H_{k+1-\delta}^{k+2}(G)}\le
C\n \bL(q)u\n_{\mathcal H_{k+1-\delta}^k(G,\pG)},\quad q\ge q_1,
\end{equation}
where $c ,C>0$ do not depend on $u$ and $q$.
\end{theorem}

\section{Nonlocal Problems in Weighted
Spaces}\label{secSolvabilityWeighted}

We fix an arbitrary orbit $\cK_\nu$ and assume that it consists of
points $g_j$, $j=1,\dots, N_\nu$. Denote by $u_j(y)$ the function
$u(y)$ for $y\in{\mathcal O}_{\varepsilon_1}(g_j)$. If
$g_j\in\overline{\Gamma_i},$ $y\in{\mathcal
O}_{\varepsilon}(g_j),$ and $\Omega_{is}(y)\in{\mathcal
O}_{\varepsilon_1}(g_k),$ then denote by $u_k(\Omega_{is}(y))$ the
function $u(\Omega_{is}(y))$. In this case, nonlocal
problem~\eqref{eq2-3} takes the following form in the
$\varepsilon$-neighborhood of the orbit $\mathcal K_\nu$:
$$
\begin{aligned}
 P u_j-q u_j=f(y),& &&  y\in\mathcal O_\varepsilon(g_j)\cap
 G,\\
u_j(y)- \sum\limits_{s=1}^{S_i} b_{is}(y) u_k(\Omega_{is}(y)) =0,&
&& y\in \mathcal O_\varepsilon(g_j)\cap\Gamma_i,\ i\in\{1\le i\le
N:\ g_j\in\overline{\Gamma_i}\},\ j=1, \dots, N_\nu.
\end{aligned}
$$

Let $y\mapsto y'(g_j)$ be the change of variables described in
Sec.~\ref{subsectStatement}, and let
 $K_j$ and $\gamma_{j\sigma}$ be the sets defined in~\eqref{eqK_jgamma_jsigma}. Set
$ K_j^\varepsilon=K_j\cap\mathcal O_\varepsilon(0),$
$\gamma_{j\sigma}^\varepsilon=\gamma_{j\sigma}\cap\mathcal
O_\varepsilon(0)$. Introduce the functions $ U_{j}(y')=u(y(y'))$
and $F_{j}(y')=f(y(y'))$ for $y'\in K_{j}^\varepsilon, $ where
$\sigma=1$ $(\sigma=2)$ if the transformation $y\mapsto y'(g_j)$
takes $\Gamma_i$ to the side $\gamma_{j1}$ ($\gamma_{j2}$) of the
angle $K_j$. Denote $y'$ by $y$ again. Then, by virtue of
Condition~\ref{condK1}, problem~\eqref{eq2-3} takes the form
\begin{equation}\label{eqPinK-eqBinK}
  P_{j}(y,D_y)U_j-qU_j=F_{j}(y), \  y\in
  K_j^\varepsilon; \qquad
  U_j(y) -\sum\limits_{k=1}^{N_\nu}\sum\limits_{s=1}^{S_{j\sigma k}}
      B_{j\sigma ks}(y)U_k({\mathcal G}_{j\sigma ks}y)
    =0, \  y\in\gamma_{j\sigma}^\varepsilon.
\end{equation}
Here $P_j(y,D_y)$ is a second-order  elliptic differential
operator with real-valued $C^\infty$ coefficients such
that the principal homogeneous part of $P_{j}(0,D_y)$ is the
Laplace operator $\Delta$; $B_{j\sigma ks}(y)$ are smooth
functions; ${\mathcal G}_{j\sigma ks}$ is an operator of rotation
by an angle~$\omega_{j\sigma ks}$ and homothety with a
coefficient~$\chi_{j\sigma ks}>0$ such that $
|(-1)^\sigma \omega_{j}+\omega_{j\sigma ks}|<\omega_{k} $.

Following~\cite{GurSkub}, we freeze the coefficients of
problem~\eqref{eqPinK-eqBinK} at the point $y=0$, replace the
operators $P_{j}(0,D_y)$ by their principal homogeneous parts, and
set $q=1$. Thus, we  consider the following problem:
\begin{equation}\label{eq14-15}
  \Delta U_j-U_j=F_{j}(y), \  y\in
  K_j; \qquad
  \cB_{j\sigma}U\equiv
  U_j(y)-\sum\limits_{k=1}^{N_\nu}\sum\limits_{s=1}^{S_{j\sigma k}}
      b_{j\sigma ks}U_k({\mathcal G}_{j\sigma ks}y)
    =0, \ y\in\gamma_{j\sigma},
\end{equation}
where $U=(U_1,\dots,U_{N_\nu})$ and $b_{j\sigma ks}=B_{j\sigma
ks}(0)$. It follows from Condition~\ref{cond1.2} that
\begin{equation}\label{eq11-13}
b_{j\sigma ks}\ge0,\qquad
\sum\limits_{k=1}^{N_\nu}\sum\limits_{s=1}^{S_{j\sigma k}}
b_{j\sigma ks}\le 1,\qquad
\sum\limits_{k=1}^{N_\nu}\Bigg(\sum\limits_{s=1}^{S_{j1k}} b_{j1
ks}+\sum\limits_{s=1}^{S_{j2k}} b_{j2 ks}\Bigg)<2.
\end{equation}

Problem~\eqref{eq14-15} should be studied in weighted spaces with nonhomogeneous weight
(cf.~\cite{GurSkub}). Denote by $E_a^k(K_j)$ the completion of the set
$C_0^\infty(\overline{K_j}\setminus\{0\})$ with respect to the norm
$$
\|v\|_{E_a^k(K_j)}=\Bigg(\sum\limits_{|\alpha|\le
k}\,\int\limits_{K_j} |y|^{2a}(|y|^{2(|\alpha|-k)}+1)|D^\alpha
v(y)|^2\, dy\Bigg)^{1/2},
$$
where $k\ge0$ is an integer and $a\in\bbR$. Denote by
$E_a^{k-1/2}(\gamma_{j\sigma})$ ($k\ge1$ is an integer) the space
of traces on $\gamma_{j\sigma}$ (with the infimum-norm). Introduce
the spaces of vector-valued functions
$$
\mathcal E_{a}^{k+2}(K)=\prod\limits_{j=1}^{N_\nu}
E_{a}^{k+2}(K_j),\quad \mathcal
E_{a}^k(K,\gamma)=\prod\limits_{j=1}^{N_\nu}\Big(E_{a}^k(K_j)
\times\prod\limits_{\sigma=1,2}
E_{a}^{k+3/2}(\gamma_{j\sigma})\Big).
$$

Consider the operator $\cL:\mathcal E_{1-\delta}^2(K)\to\mathcal
E_{1-\delta}^0(K,\gamma)$ given by $ \cL U=\{\Delta U_j-U_j,\
\cB_{j\sigma}U\}. $

Our aim is to prove that the operator $\cL$ is an isomorphism for
all sufficiently small $\delta\ge 0$. To this end, we consider the
 analytic operator-valued function $$
\begin{gathered}
 \tilde{\mathcal
L}(\lambda):\prod\limits_{j=1}^{N_\nu}
W_2^{2}(-\omega_j,\omega_j)\to\prod\limits_{j=1}^{N_\nu} (L_2(-\omega_j,
\omega_j)\times\mathbb C^2),\\
 \tilde{\mathcal
L}(\lambda)\varphi=\Big\{\varphi_j''-\lambda^2\varphi_j,\
  \varphi_j((-1)^\sigma\omega_j)-\sum\limits_{k,s} (\chi_{j\sigma ks})^{i\lambda}
              b_{j\sigma ks}\varphi_k((-1)^\sigma \omega_j+\omega_{j\sigma
              ks})\Big\}.
\end{gathered}
$$

\begin{lemma}\label{l1}
Let Conditions~$\ref{cond1.1}$--$\ref{cond1.2}$ hold. Then the
line $\Im\lambda=0$ contains no eigenvalues of $\tilde{\mathcal
L}(\lambda)$.
\end{lemma}
\begin{proof}
1. We assume that $\lambda_0\ne0$ (the case $\lambda_0=0$ is
analogous but simpler) is an eigenvalue of $\tilde{\mathcal
L}(\lambda)$ and $\lambda_0$ is a real number. Let $\phi(\omega)$
be the corresponding eigenvector. We  represent it in the form $
\phi(\omega)=\phi^1(\omega)+i\phi^2(\omega), $ where
$\phi^1(\omega)$ and $\phi^2(\omega)$ are real-valued $C^\infty$
functions. It is easy to see that the function $
U=r^{i\lambda_0}\phi(\omega)=e^{i\lambda_0\ln r}\phi(\omega) $ is
a solution of the following problem:
\begin{equation}\label{eqPr23'}
  \Delta U_j =0,\ y\in
  K_j;\ \qquad
  \cB_{j\sigma}U     =0, \  y\in\gamma_{j\sigma}.
\end{equation}
 We represent the function
$U$ in the form $U=V+iW$, where $ V=\cos(\lambda_0\ln
r)\phi^1(\omega)-\sin(\lambda_0\ln r)\phi^2(\omega),$ $
W=\cos(\lambda_0\ln r)\phi^2(\omega)+\sin(\lambda_0\ln
r)\phi^1(\omega)$. Since the coefficients in~\eqref{eqPr23'} are
real, it follows that $V$ (as well as $W$) is a solution of the
problem
\begin{equation}\label{eq18-19}
  \Delta V_j=0, \ y\in
  K_j;\qquad
  \cB_{j\sigma}V =0,\ y\in\gamma_{j\sigma}.
\end{equation}

Denote $ M=\max_{j=1,\dots,N_\nu}\sup_{y\in K_j}|V_j(y)|. $ We
claim that $M=0$. Assume the contrary:  $M>0$.

2. If $|V_j(y^0)|=M$ for some $j$ and $y^0\in K_j$, then
$V_j(y)\equiv M$ by  the maximum principle, and  the nonlocal
conditions in~\eqref{eq18-19} imply
\begin{equation}\label{eq20}
M=|V_j(y^0)|=|V_j|_{\gamma_{j\sigma}}|\le
M\sum\limits_{k,s}b_{j\sigma ks},\qquad \sigma=1,2.
\end{equation}
However, $ 0\le\sum\limits_{k,s}b_{j\sigma ks}<1 $ for $\sigma=1$ or $2$ due to
conditions~\eqref{eq11-13}, which contradicts~\eqref{eq20}.

3. Let $|V_j(y_0)|=M$ for some $j$, $\sigma=1$ or $2$, and $y_0\in\gamma_{j\sigma}$. In
this case, taking into account~\eqref{eq11-13}, we again deduce from the nonlocal
conditions in~\eqref{eq18-19} that
\begin{equation}\label{eq21}
M=|V_j(y^0)|\le\sum\limits_{k,s}b_{j\sigma ks}|V_k(\cG_{j\sigma
ks}y^0)|\le M
\end{equation}
for  $\sigma=1$ or $ 2$. Therefore, the inequalities in~\eqref{eq21} reduce to
equalities, and we see that $\sum\limits_{k,s}b_{j\sigma ks}=1$ and $|V_k(\cG_{j\sigma
ks}y^0)|=M$ for at least one pair $(k,s)$. However, this contradicts what has been
already proved, since $\cG_{j\sigma ks}y^0\in K_k$.

4. Finally, we assume that there is a sequence
$\{y^s\}_{s=1}^\infty\subset K_j$ such that $|V_j(y^s)|\to M$ for some $j$ as
 $|y^s|\to 0$ or $|y^s|\to \infty$.

We note that the function $V_j$ is periodic with respect to $\ln r$, i.e., the function
$V_j$ is completely defined by its values on the set $\hat
K_j=\overline{K_j}\cap\big\{1\le r\le e^{2\pi/|\lambda_0|}\big\}. $

Since the set $\hat K_j $ is a compact, there is a sequence $\{\hat
y^s\}_{s=1}^\infty\subset \hat K_j$ such that $|V_j(\hat y^s)|\to M$
as $\hat y^s\to\hat y$, where $\hat y\in \hat K_j$. It follows from
the continuity of the function $V_j(y)$ on the compact $\hat K_j $
that $|V_j(\hat y)|= M$. However, this is impossible due to what
has been proved above.

5. It follows from items 1--4 that $M=0$, hence $V=0$, i.e.,
$\phi^1(\omega)=\phi^2(\omega)=0$.
\end{proof}

\begin{lemma}\label{l2}
Let Conditions~$\ref{cond1.1}$--$\ref{cond1.2}$ hold. Then the
operator $\cL:\mathcal E_1^2(K)\to\mathcal E_1^0(K,\gamma)$ is an
isomorphism.
\end{lemma}
\begin{proof}
1. First, we show that the operator $\cL:\mathcal
E_1^2(K)\to\mathcal E_1^0(K,\gamma)$ has the Fredholm property and
$\ind\cL=0$. Consider the family of operators $\cL_t:\mathcal
E_1^2(K)\to\mathcal E_1^0(K,\gamma)$ given by $ \cL_t
U=\Big\{\Delta U_j-U_j,\
U_j|_{\gamma_{j\sigma}}-t\sum\limits_{k,s}
      b_{j\sigma ks}U_k({\mathcal G}_{j\sigma ks}y)|_{\gamma_{j\sigma}}\Big\},\
      0\le t\le1.
$ Similarly to the operator $\tilde{\cL}(\lambda)$, we introduce
the operators $\tilde{\cL}_t(\lambda)$. By Lemma~\ref{l1}, the
operators $\tilde{\cL}_t(\lambda)$ have no eigenvalues on the line
$\Im\lambda=0$. Therefore, the operators $\cL_t$ have the Fredholm
property due to Theorem~9.1 in~\cite{GurGiess}. By using the
homotopy stability of the index of Fredholm operators, we obtain
$\ind\cL_t=\const$ for $t\in[0,1]$. Since the local operator
$\cL_0$ is an isomorphism (see, e.g., Sec.~10.3
in~\cite{GurGiess}), it follows that $\ind\cL=\ind\cL_0=0$.

2. It remains to prove that $\dim\ker\cL=0$. Let $U\in \mathcal
E_1^2(K)$ be a real-valued solution of the problem
\begin{equation}\label{eq22-23}
\Delta U_j=U_j, \ y\in
  K_j; \qquad
\cB_{j\sigma}U =0,\  y\in\gamma_{j\sigma}.
\end{equation}
Due to the interior regularity theorem, the functions $U_j$ are
infinitely differentiable in $K_j$. Let us prove that $U_j$ are
continuous on $\overline{K_j}$.

Since the line $\Im\lambda=0$ contains no eigenvalues of
$\tilde{\cL}(\lambda)$, it follows from~\cite{SkDu90} that there
is a number $\delta\in[0,1]$ such that the strip
$-1-\delta\le\Im\lambda\le 0$ contains finitely many eigenvalues
$\{\lambda_k\}$ of $\tilde{\cL}(\lambda)$ and
$-1-\delta<\Im\lambda_k<0$. Taking into account that $U_j\in
E_1^2(K_j)\subset H_1^2(K_j)$ is a solution of problem
\eqref{eq22-23} with the right-hand sides $U_j\in
E_1^2(K_j)\subset H_{-\delta}^0(K_j)$ and applying Theorem~2.2
in~\cite{GurPetr03} (about the asymptotics of solutions for
nonlocal problems), we obtain
\begin{equation}\label{eq23''}
U=\sum\limits_{k}\sum\limits_{q=1}^{J_k}\sum\limits_{m=0}^{\varkappa_{qk}-1}c_k^{(m,q)}W_k^{(m,q)}
+U',\qquad
W_k^{(m,q)}(\omega,r)=r^{i\lambda_k}\sum\limits_{l=0}^m\dfrac{1}{l!}(i\ln
r)^l\varphi_k^{(m-l,q)}(\omega),
\end{equation}
where
$\varphi_k^{(0,q)},\dots,\varphi_k^{(\varkappa_{qk}-1,q)}\in\prod\limits_j
C^\infty([-\omega_j,\omega_j])$ is the Jordan chain corresponding
to the eigenvalue $\lambda_k$, $c_k^{(m,q)}$ are constants, and
$U'_j\in H_{-\delta}^2(K_j)$. Thus, using the Sobolev embedding
theorem, we see that the functions $U_j$ are continuous on
$\overline{K_j}$ and $U_j(0)=0$.

Furthermore, we claim that
\begin{equation}\label{eq23'''}
|U_j(y)|\to 0\qquad \text{as}\quad |y|\to\infty.
\end{equation}
 Indeed, since $U\in \mathcal E_1^2(K)$, it
follows that $U\in \mathcal E_1^0(K)$. Combining this with the
fact that $U$ is a solution of homogeneous problem~\eqref{eq22-23}
and applying Theorem~3.2 in~\cite{GurGiess}, we obtain $U\in
\mathcal E_3^2(K)$. Fixing an arbitrary $a\ge 1$ and repeating
these arguments finitely many times, we have $U\in \mathcal
E_a^2(K)$. Setting $V(\omega,r)=U(\omega,r^{-1})$ and using the
Sobolev embedding theorem and the fact that $a$ can be arbitrarily
large, we see that the functions $V_j(y)$ are continuous at the
origin and $|V_j(y)|\to 0$ as $|y|\to0$. This
implies~\eqref{eq23'''}

3. Set $M=\max\limits_{j=1,\dots,N_\nu} \sup\limits_{y\in
\overline{K_j}} |U_j(y)|$. We claim that $M=0$. Assume the
contrary; let $M>0$. Due to the above properties of $U_j$, each of
the functions $|U_j(y)|$ achieves its maximum at some point
$y_0\in\overline{K_j}\setminus\{0\}$. If $|U_j(y_0)|=M$ for some
$j$ and $y_0\in K_j$, then $U_j(y)\equiv\const$ by the maximum
principle. In this case, using the equation in~\eqref{eq22-23}, we
obtain $M\equiv |U_j|=|\Delta U_j|=0$.

If $|U_j(y_0)|=M$ for $y_0\in \gamma_{j\sigma}$, where $\sigma=1$ or $ 2$, then, using
the nonlocal conditions in~\eqref{eq22-23} and inequalities~\eqref{eq11-13}, we obtain
\begin{equation}\label{eq25}
M=|U_j(y_0)|\le \sum\limits_{k,s}b_{j\sigma ks}|U_k(\cG_{j\sigma
ks}y_0)|\le M.
\end{equation}
Thus, the inequalities in~\eqref{eq25} becomes the equalities,
which implies $\sum\limits_{k,s}b_{j\sigma ks}=1$ and
$|U_k(\cG_{j\sigma ks}y_0)|=M$ for at least one pair  $(k,s)$.
However, $\cG_{j\sigma ks}y_0\in K_k$, which is impossible by what
has been proved above.
\end{proof}

\begin{corollary}\label{cor1}
Let Conditions~$\ref{cond1.1}$--$\ref{cond1.2}$ hold. Then there exists a number
$\delta_1>0$ such that  the operator $\cL:\mathcal E_{k+1-\delta}^{k+2}(K)\to\mathcal
E_{k+1-\delta}^k(K,\gamma)$, $k=0,1,2,\dots$, is an isomorphism for $
\delta\in[0,\delta_1]$.
\end{corollary}
\begin{proof}
By Lemma~\ref{l2}, the operator $\cL:\mathcal E_{1}^2(K)\to\mathcal E_{1}^0(K,\gamma)$ is
an isomorphism. On the other hand, it follow from Lemma~\ref{l1} and from the
discreteness of the spectrum of $\tilde L(\lambda)$ (see~\cite{SkDu90}), there exists a
number  $\delta_1>0$ such that  the strip $-\delta_1\le\Im\lambda\le 0$ contains no
eigenvalues of $\tilde{\cL}(\lambda)$. Similarly to Proposition 2.8 in~\cite[Chap.
8]{NP}, one can show that  the operator $\cL:\mathcal E_{1-\delta}^2(K)\to\mathcal
E_{1-\delta}^0(K,\gamma)$ is an isomorphism for  $\delta\in[0,\delta_1]$. Due to Theorems
9.2 and 9.3 in~\cite{GurGiess}, the operator $\cL:\mathcal
E_{k+1-\delta}^{k+2}(K)\to\mathcal E_{k+1-\delta}^k(K,\gamma)$ is also an isomorphism.
\end{proof}

\begin{proof}[Proof of Theorem $\ref{th1}$]
The required assertion follows from Theorem~8.1 in~\cite{GurSkub} and
Corollary~\ref{cor1}.
\end{proof}

\section{Nonlocal Problems in Spaces of Continuous
Functions}\label{subsectNonlocalProblemsInC}

In what follows, we assume that a number $\delta\in[0,1]$ is fixed
in such a way that neither the strip $-\delta\le\Im\lambda\le 0$
nor the line $\Im\lambda=-1-\delta$ contain  an eigenvalue of
$\tilde{\cL}(\lambda)$. The existence of such a number follows
from Lemma~\ref{l1} and from the discreteness of the spectrum of $\tilde L(\lambda)$
(see~\cite{SkDu90}).

Let $q_1$ be the number occurring in Theorem~\ref{th1}. First,
we construct an analog of the barrier function for nonlocal
problems. Consider the following auxiliary  problem:
\begin{equation}\label{eq26-27}
Pv-q_1 v =0,\  y\in G;\qquad v|_{\Gamma_i}-\bB_i v =1,\
y\in\Gamma_i,\ i=1,\dots,N.
\end{equation}

\begin{lemma}\label{l3}
Let Conditions~$\ref{cond1.1}$--$\ref{cond1.2}$ hold. Then problem
\eqref{eq26-27} admits a bounded solution $ v\in C^\infty
(\overline{G}\setminus\cK) $ such that
$\inf\limits_{y\in\overline{G}\setminus\cK}v(y)>0$.
\end{lemma}
\begin{proof}
1. We fix an arbitrary orbit $\cK_\nu$ and consider the model
problem
\begin{equation}\label{eqVstavka13_1-2}
  \Delta W_j^1 =0,\  y\in
  K_j^\varepsilon;\qquad
  W_j^1(y)-\sum\limits_{k,s}
      b_{j\sigma ks}W_k^1({\mathcal G}_{j\sigma ks}y)
     =1, \  y\in\gamma_{j\sigma}^\varepsilon.
\end{equation}
Let us
search a solution of problem~\eqref{eqVstavka13_1-2} in the form
\begin{equation}\label{eqVstavka13_3}
W_j^1=\varphi_j(\omega),\qquad |\omega|<\omega_j,\ j=1,\dots,N_\nu.
\end{equation}
Clearly, the functions $ \varphi_1(\omega),\dots ,\varphi_{N_\nu}(\omega) $ must satisfy
the relations
\begin{equation}\label{eq32-33}
  \varphi_j''(\omega)=0,\  |\omega|<\omega_j;\qquad
  \varphi_j((-1)^\sigma\omega_j)-\sum\limits_{k,s}
      b_{j\sigma ks}\varphi_k((-1)^\sigma\omega_j+\omega_{j\sigma ks})
    =1,
\end{equation}
or, equivalently, $ \tilde{\cL}(0)\varphi=\{\tilde F_j,\tilde
F_{j\sigma}\},\ \tilde F_j=0,\ \tilde F_{j\sigma}=1. $ Due to
Lemma~\ref{l1}, the number $\lambda=0$ is not an eigenvalue of
$\tilde{\cL}(\lambda)$. Since the operator $\tilde{\cL}(\lambda)$ has the Fredhom property and its index equals zero~\cite{SkDu90}, there exists a unique
(real-valued) solution $\varphi\in\prod\limits_j
C^\infty([-\omega_j,\omega_j])$ of problem~\eqref{eq32-33}.
Clearly, $\varphi_j(\omega)$ are linear functions. Using the
nonlocal conditions in~\eqref{eq32-33} and
relations~\eqref{eq11-13}, one can check that
$\varphi_j(\omega)>0$ for $\omega\in[-\omega_j,\omega_j]$.

2. Consider a   function $\xi\in C^\infty(\bbR^2)$ such that
$\xi(y)=1$ for $y\in\cO_{\varepsilon/2}(\cK)$ and $\supp\xi\subset
\cO_\varepsilon(\cK)$.

Let us search the solution $v$ of the original problem
\eqref{eq26-27} in the form
\begin{equation}\label{eqVstavka13_4}
v(y)=w^1(y)+v^1(y),\qquad y\in G,
\end{equation}
where $w^1(y)=\xi(y)W_j^1(y'(y))$, $y\in\cO_\varepsilon(g_j)$,
$g_j\in\cK_\nu$, $y'\mapsto y(g_j)$ is the transformation inverse
to the transformation $y\mapsto y'(g_j)$ from
Sec.~\ref{subsectStatement}, and the function $w^1$ is extended by
zero to $G\setminus\cO_\varepsilon(\cK)$; the function $v^1$ is
unknown.

It follows from relations~\eqref{eq26-27}  and
\eqref{eqVstavka13_4} that the function $v^1$ satisfies the
relations
\begin{equation}\label{eqVstavka13_5-6}
Pv^1-q_1v^1 =f^1(y),\  y\in G; \qquad v^1|_{\Gamma_i}-\bB_i  v^1
=f_i^1(y),\ y\in\Gamma_i,
\end{equation}
where
\begin{equation}\label{eqVstavka13_7}
f^1=-Pw^1+q_1w^1,\qquad
f_i^1=1-w^1|_{\Gamma_i}+\bB_i w^1|_{\Gamma_i}.
\end{equation}

Set $V_j^1(y')=v^1(y(y'))$, $F_j(y')=f^1(y(y'))$, and
$F_{j\sigma}(y')=f_i^1(y(y'))$, $y'\in K_j^\varepsilon$, where
$y\mapsto y'(g_j)$ is the transformation from
Sec.~\ref{subsectStatement},
$g_j\in\cK_\nu\cap\overline{\Gamma_i}$. Denote $y'$ by $y$ again.
Then, due to~\eqref{eqVstavka13_1-2}  and~\eqref{eqVstavka13_7},
we have
\begin{equation}\label{eqVstavka13_8}
F_j(y)=(\Delta -P_j(y,D_y))W_j^1+q_1W_j^1,\quad
F_{j\sigma}(y)=\sum\limits_{k,s}(B_{j\sigma ks}(y)-b_{j\sigma
ks})W_k^1(\cG_{j\sigma ks}y),\quad y\in K_j^{\varepsilon/2},
\end{equation}
where $P_j(y,D_y)$ and $B_{j\sigma ks}(y)$ are the same as
in~\eqref{eqPinK-eqBinK}.

Using the facts the the principal homogeneous part of the operator
$P_j(0,D_y)$ is the Laplace operator and $B_{j\sigma
ks}(0)=b_{j\sigma ks}$ and applying the Taylor formula, we deduce
from representation~\eqref{eqVstavka13_3} and
relations~\eqref{eqVstavka13_8} that $F_j\in
H_{k+1-\delta}^k(K_j^{\varepsilon/2})$ and $F_{j\sigma}\in
H_{k+1-\delta}^{k+3/2}(\gamma_{j\sigma}^{\varepsilon/2})$, i.e., $
\{f^1,f_i^1\}\in\cH_{k+1-\delta}^k(G,\pG). $ Therefore, by
Theorem~\ref{th1},  there is a unique solution $v^1\in
H_{k+1-\delta}^{k+2}(G)$ of problem~\eqref{eqVstavka13_5-6}. Since
$k\ge0$ is arbitrary, it follows from the Sobolev embedding
theorem that the function $v$ given by~\eqref{eqVstavka13_4}
belongs to $C^\infty(\oG\setminus\cK)$. Clearly, it  is a solution
of the original problem~\eqref{eq26-27}.

3. Let us prove that $v^1\in C(\oG)$ and $v^1(y)=0$ for $y\in
\cK$. Due to~\eqref{eqVstavka13_5-6}, the functions $V_j^1(y)$
satisfy the following relations:
\begin{equation}\label{eqVstavka13_10-11}
\Delta V_j^1 =F_j^1(y)+F_j(y),\ y\in K_j^{\varepsilon/2};  \qquad
V_j^1(y)-\sum\limits_{k,s}b_{j\sigma ks} V_k^1(\cG_{j\sigma ks}y)
=F_{j\sigma}^1(y)+F_{j\sigma}(y),\
y\in\gamma_{j\sigma}^{\varepsilon/2},
\end{equation}
where $F_j^1=(\Delta -P_j(y,D_y))V_j^1+q_1V_j^1$ and
$F_{j\sigma}^1=\sum\limits_{k,s}(B_{j\sigma ks}(y)-b_{j\sigma
ks})V_k^1(\cG_{j\sigma ks}y).$

Using the facts that the principal homogeneous part of the
operator $P_j(0,D_y)$ is the Laplace operator and $B_{j\sigma
ks}(0)=b_{j\sigma ks}$ and applying the Taylor formula once more,
we represent the right-hand sides of problem
\eqref{eqVstavka13_10-11} as follows:
\begin{equation}\label{eqVstavka13_12}
F_j^1+F_j=F_j^1+F_j^2+r^{-1}\psi_j(\omega),\qquad
F_{j\sigma}^1+F_{j\sigma}=F_{j\sigma}^1+F_{j\sigma}^2+\psi_{j\sigma}r,
\end{equation}
where $\psi_j\in C^\infty([-\omega_j,\omega_j])$, $F_j^1+F_j^2\in
H_{-\delta}^0(K_j^{\varepsilon/2})$ and $\psi_{j\sigma}\in\bbR$,
$F_{j\sigma}^1+F_{j\sigma}^2\in
H_{-\delta}^{3/2}(\gamma_{j\sigma}^{\varepsilon/2})$.

To obtain the asymptotics of the functions $V_j^1$, we denote by
$\{\lambda_k\}$ a (finite) set of eigenvalues of
$\tilde\cL(\lambda)$ lying in the strip
$-1-\delta<\Im\lambda<-\delta$. Then Theorem~2.2
in~\cite{GurPetr03} and Lemma~4.3 in~\cite{GurPetr03} applied to
problem~\eqref{eqVstavka13_10-11} with right-hand
side~\eqref{eqVstavka13_12} imply that
\begin{equation}\label{eqVstavka13_13}
V^1=r\sum\limits_{l=0}^\varkappa\dfrac{1}{l!}(i\ln
r)^lu^{(l)}(\omega)+\sum\limits_k\sum\limits_{q=1}^{J_k}\sum\limits_{m=0}^{\varkappa_{qk}-1}c_k^{(m,q)}W_k^{(m,q)}+V^2,\quad
y\in K_j^{\varepsilon/2},
\end{equation}
where $u^{(l)}\in\prod\limits_j C^\infty([-\omega_j,\omega_j])$,
the functions $W_k^{(m,q)}$ are of the same form as
in~\eqref{eq23''}, $c_k^{(m,q)}$ are some constants, and $V_j^2\in
H_{-\delta}^2(K_j^{\varepsilon/2})$. The asymptotic
formula~\eqref{eqVstavka13_13} and the Sobolev embedding theorem
imply that $V_j^1\in C(\overline{K_j^{\varepsilon/2}})$ and
$V_j^1(0)=0$. Therefore, $v^1\in C(\oG)$ and $v^1(0)=0$ for
$y\in\cK$. In particular, this means that the function $v=v^1+w^1$
is bounded.

4. It remains to show that $m>0$, where
$m=\inf\limits_{y\in\overline{G}\setminus\cK}v(y)$. Assume the
contrary; let $m\le 0$. Let  a sequence
$\{y^k\}\subset\overline{G}\setminus\cK$ be  such that $v(y^k)\to
m$ as $k\to\infty$. Since the sequence $\{y^k\}$ is bounded, it
contains a convergent subsequence (which we also denote by
$\{y^k\}$). Let $y^k\to y^0$ as $k\to\infty$, where
$y^0\in\overline G$.

Using the maximum principle, the nonlocal conditions
in~\eqref{eq26-27}, and relations~\eqref{eq4-5}, one can verify
that $y_0\notin\oG\setminus\cK$. Assume that $y^0\in\cK_\nu$ for
some $\nu$. It follows from what has been proved in item 1 that
there is a constant $A>0$ such that $w^1(y)\ge A$ in some
neighborhood of $y^0$ (excluding the point $y^0$ itself, at which
the function $w^1$ need not be defined). On the other hand, we
have proved in item 3 that $v^1(y^0)=0$. Therefore, $v(y)\ge A/2$
in some neighborhood of $y^0$ (excluding the point $y^0$ itself).
Thus, the sequence $\{v(y^k)\}$ cannot converge to the nonpositive
number $m$.
\end{proof}

For any closed set   $Q\subset\oG$ such that  $Q\cap \cK\ne\varnothing$, we introduce the
space
\begin{equation}\label{eqC_K}
C_\cK(Q)=\{u\in C(Q): u(y)=0,\ y\in Q\cap \cK\}
\end{equation}
with the maximum-norm.

Consider the space of vector-valued functions $ \cC_\cK(\pG)=\prod\limits_{i=1}^N
C_\cK(\overline{\Gamma_i}) $ with the norm $$ \|\psi\|_{\cC_\cK(\pG)}=
\max\limits_{i=1,\dots,N}\max\limits_{y\in\overline{\Gamma_i}}\|\psi_i\|_{C(\overline{\Gamma_i})},
$$ where $\psi=\{\psi_i\}$, $\psi_i\in C_\cK(\overline{\Gamma_i})$.

We study the solvability of the problem
\begin{equation}\label{eq47-48}
Pu-q u =0, \  y\in G;\qquad u|_{\Gamma_i}-\bB_i u=\psi_i(y), \ y\in\Gamma_i,\
i=1,\dots,N,
\end{equation}
in the space of continuous functions.

\begin{lemma}\label{l-th2}
Let Conditions~$\ref{cond1.1}$--$\ref{cond1.2}$ be fulfilled, and let $q\ge q_1$. Then,
for any $\psi=\{\psi_i\}\in \cC_\cK(\pG)$, there exists a unique solution $u\in
C_\cK(\overline G)\cap C^\infty(G)$ of problem~\eqref{eq47-48}. Furthermore,   the
following estimate holds{\rm:}
\begin{equation}\label{eq49}
\|u\|_{C(\overline G)}\le c_1\|\psi\|_{\cC_\cK(\pG)},
\end{equation}
where $c_1>0$ does not depend on  $\psi$ and $q$.
\end{lemma}
\begin{proof}
1. We prove the lemma for the functions $\psi_i$ which are infinitely differentiable and
vanish in a neighborhood of the sets $\overline{\Gamma_i}\cap\cK$. The general case will
follow by the limit passage. Given $\psi_i$ with the above properties, we have $\psi_i\in
H_{-\delta}^{3/2}(\Gamma_i)$. Therefore, by Theorem~\ref{th1}, there exists a unique
solution $u\in H_{1-\delta}^2(G)$ of problem \eqref{eq47-48}. By Lemma 5.1
in~\cite{GurRJMP04}, $u\in C^\infty(\overline{G}\setminus\cK)$. Let $\{\lambda_k\}$ be a
(finite) set of eigenvalues of $\tilde\cL(\lambda)$ lying in the strip
$-1-\delta<\Im\lambda<-\delta$. Then, due  to Theorem~2.2 in~\cite{GurPetr03} (about the
asymptotics of solutions for nonlocal problems), the function $u$ has the following
asymptotics near an arbitrary point $g_j\in\cK_\nu$ ($j=1,\dots,N_\nu$,
$\nu=1,\dots,N_0$):
$$
u(y)=\sum\limits_{k}\sum\limits_{q=1}^{J_k}\sum\limits_{m=0}^{\varkappa_{qk}-1}c_k^{(m,q)}W_{kj}^{(m,q)}
+u'(y), \qquad y\in G\cap\cO_\varepsilon(g_j),
$$
where  $c_k^{(m,q)}$ are constants, the functions
$W_{kj}^{(m,q)}(\omega,r)$ are of the same form as the components
of the vector $W_{k}^{(m,q)}(\omega,r)$ in~\eqref{eq23''}
($\omega,r$ are the polar coordinate centered at the point $g_j$),
and $u'\in H_{-\delta}^2(G)$. Therefore, applying the Sobolev
embedding theorem, we see that $u\in C(\overline G)$ and
\begin{equation}\label{eq50}
 u(y)=0,\qquad y\in\cK.
\end{equation}

2. Let us prove estimate~\eqref{eq49}. Set $
M=\|\psi\|_{\cC_\cK(\pG)} $ and  assume that $M>0$.

Set $ w_\pm(y)=Mv(y)\pm u(y), $ where $v(y)$ is the function from
Lemma~\ref{l3}. Equalities \eqref{eq26-27} and \eqref{eq47-48}
imply that the functions $w_\pm$ satisfy the relations
$$
Pw_\pm-q w_\pm =M(q_1-q)v(y),\  y\in G;\qquad
w_\pm|_{\Gamma_i}-\bB_i w_\pm =M\pm\psi_i(y),\ y\in\Gamma_i,\
i=1,\dots,N.
$$
Since $q_1\le q$, $v(y)>0$, $y\in G$ (by Lemma~\ref{l3}), and
$M\ge\pm\psi_i$, it follows that
\begin{equation}\label{eqVst62'-62''}
Pw_\pm-q w_\pm \le0,\  y\in G,\qquad
w_\pm|_{\Gamma_i}-\bB_i w_\pm \ge0,\  y\in\Gamma_i,\ i=1,\dots,N.
\end{equation}

We claim that $m_\pm=\inf\limits_{y\in\overline
G\setminus\cK}w_\pm(y)\ge0$. Assume the contrary; let  $m_\pm<0$.
As in  item 4 of the proof of Lemma~\ref{l3}, we consider a
sequence $\{y^k\}\subset \oG\setminus\cK$ such that $y^k\to y^0$
and $w_\pm(y_k)\to m_\pm$ as $k\to \infty$, where $y^0\in\overline
G$. The following three cases are possible: $y^0\in G$,
$y^0\in\Gamma_i$ for some $i$, and $y^0\in \cK$.

Let $y^0\in G$. Since $w_\pm(y)$ is continuous in $G$, we see that
it achieves its negative minimum $m$ inside the domain. It follows
from the first inequality in~\eqref{eqVst62'-62''} and from   the
maximum principle that $w_\pm(y)=m_\pm$ for $y\in G$. Combining
this relation with Condition~\ref{cond1.1}, we obtain
$Pw_\pm(y^0)-qw_\pm(y^0)=p_0(y^0)m_\pm-qm_\pm\ge -qm_\pm>0$, which
contradicts the first inequality in~\eqref{eqVst62'-62''}.

Let $y^0\in\Gamma_i$ for some $i$. In this case, it follows
from~\eqref{eqVst62'-62''} and~\eqref{eq4-5} that
\begin{equation}\label{eq54}
m_\pm=w_\pm(y^0)\ge \sum\limits_{s=1}^{S_i}
b_{is}(y^0)w_\pm(\Omega_{is}(y^0))\ge m_\pm\sum\limits_{s=1}^{S_i}
b_{is}(y^0)\ge m_\pm.
\end{equation}
Therefore, the inequalities in~\eqref{eq54} are in fact equalities. This means that
$\sum\limits_{s=1}^{S_i} b_{is}(y^0)=1$ and $w_\pm(\Omega_{is}(y^0))=m_\pm$ for some $s$,
i.e., the function $w_\pm(y)$ achieves its negative minimum at the interior point
$\Omega_{is}(y^0)\in G$. This contradicts what has been proved above.

Finally, assume that $y^0\in\cK_\nu$ for some $\nu$. By Lemma
\ref{l3} we have
$m=\inf\limits_{y'\in\oG\setminus\cK}v(y')>0$, which yields
$$
Mv(y)\ge Mm>0,\qquad y\in\overline G\setminus\cK.
$$
It follows from the latter inequality and from~\eqref{eq50} that
$$
w_\pm(y)=Mv(y)\pm u(y)\ge Mm/2>0
$$
in some neighborhood of $y^0$ (excluding the point $y^0$ itself,
where $w_\pm(y)$ need not be defined). Therefore, the sequence
$\{w_\pm(y^k)\}$ cannot converge to the negative number $m_\pm$.

Thus, we have proved that $\inf\limits_{y\in\overline
G\setminus\cK}w_\pm(y)\ge0$, which yields
$$
|u(y)|\le M v(y)\le M\sup\limits_{y'\in\overline
G\setminus\cK}v(y'),\qquad y\in \overline G\setminus\cK.
$$
Since the function $u(y)$ is continuous in $\overline G$, the last
inequality implies estimate~\eqref{eq49}, where
$c_1=\sup\limits_{y'\in\overline G\setminus\cK}v(y')$. Clearly, the constant
 $c_1>0$ does not depend on $\psi$ and $q$.
\end{proof}

Now we consider the problem
\begin{equation}\label{eq47-48'}
Pu-q u =f(y), \  y\in G;\qquad u|_{\Gamma_i}-\bB_i u=\psi_i(y), \ y\in\Gamma_i,\
i=1,\dots,N.
\end{equation}

\begin{theorem}\label{th2}
Let Conditions~$\ref{cond1.1}$--$\ref{cond1.2}$ be fulfilled, and let $q\ge q_1$. Then,
for any $f\in C(\oG)$ and $\psi=\{\psi_i\}\in \cC_\cK(\pG)$, there exists a unique
solution $u\in  C_\cK(\overline G)\cap W_{2,\loc}^2(G)$ of problem~\eqref{eq47-48'}.
Furthermore, if $f=0$, then $u\in   C_\cK(\oG)\cap C^\infty(G)$ and  the following
estimate holds{\rm:}
$$
\|u\|_{C(\overline G)}\le c_1\|\psi\|_{\cC_\cK(\pG)},
$$
where $c_1>0$ does not depend on  $\psi$ and $q$.
\end{theorem}
\begin{proof}
Due to Lemma~\ref{l-th2}, it suffices to prove the existence of a solution  $u\in
C_\cK(\overline G)\cap W_{2,\loc}^2(G)$ for problem~\eqref{eq47-48'} with $f\in C(\oG)$
and $\psi_i=0$. Since $f\in C(\oG)\subset H_{-\delta}^0(G) $, it follows from
Theorem~\ref{th1} that there is a unique solution  $u\in H_{1-\delta}^2(G)$ of
problem~\eqref{eq47-48'} with the right-hand sides  $\psi_i=0$. By Lemma~5.1
in~\cite{GurRJMP04} $u\in W_2^2(G\setminus\overline{\cO_\sigma(\cK)})$ for all
$\sigma>0$.

Let $\{\lambda_k\}$ be a (finite) set of eigenvalues of $\tilde\cL(\lambda)$ lying in the
strip $-1-\delta<\Im\lambda<-\delta$. Then, due  to Theorem~2.2 in~\cite{GurPetr03}
(about the asymptotics of solutions for nonlocal problems), the function $u$ has the
following asymptotics near an arbitrary point $g_j\in\cK_\nu$ ($j=1,\dots,N_\nu$,
$\nu=1,\dots,N_0$):
$$
u(y)=\sum\limits_{k}\sum\limits_{q=1}^{J_k}\sum\limits_{m=0}^{\varkappa_{qk}-1}c_k^{(m,q)}W_{kj}^{(m,q)}
+u'(y), \qquad y\in G\cap\cO_\varepsilon(g_j),
$$
where  $c_k^{(m,q)}$ are constants, the functions $W_{kj}^{(m,q)}(\omega,r)$ are of the
same form as the components of the vector $W_{k}^{(m,q)}(\omega,r)$ in~\eqref{eq23''}
($\omega,r$ are the polar coordinate centered at the point $g_j$), and $u'\in
H_{-\delta}^2(G)$. Therefore, applying the Sobolev embedding theorem, we see that $u\in
C_\cK(\overline G)$.
\end{proof}

\section{Existence of Feller semigroups}\label{secFeller}

We introduce the space   $$ C_B(\oG)=\{u\in C_\cK(\oG): u|_{\Gamma_i}-\bB_i u=0, \
y\in\Gamma_i,\ i=1,\dots,N\}. $$ We prove in this section that the unbounded operator
$\bP_B: \Dom(\bP_B)\subset C_B(\overline G)\to C_B(\overline G)$ given by
\begin{equation}\label{eqbP_BBoundedPert}
\bP_B u=Pu,\quad   \Dom(\bP_B)=\{u\in C_B(\oG)\cap W^2_{2,\loc}(G): Pu\in C_B(\overline
G)\},
\end{equation}
is a generator of a Feller semigroup.

\begin{remark}
Consider a nontransversal nonlocal condition of the form
(cf.~\cite{Ventsel,Taira1,Taira3,SkDAN89, SkRJMP95,GalSkMs,GalSkJDE})
\begin{equation}\label{eqNonlocalGeneral1}
b(y)u(y)+\int\limits_\oG[u(x)-u(\eta)]m(y,d\eta)=0,\quad y\in\pG,
\end{equation}
where  $b(y)\ge0$, $m(y,\cdot)$ is a nonnegative Borel measure, and $b(y)+m(y,\oG)>0$,
$y\in\pG$.

Introduce a nonnegative Borel measure $\mu(y,\cdot)=m(y,\cdot)/[b(y)+m(y,\oG)]$. Then the
nonlocal condition~\eqref{eqNonlocalGeneral1} can be written as follows:
\begin{equation}\label{eqNonlocalGeneral2}
u(y) - \int\limits_\oG u(\eta)\mu(y,d\eta)=0,\quad y\in\pG.
\end{equation}

Assume that $\mu(y,\cdot)=0$ for  $y\in\cK$ and  $\mu(y,\cdot)$ is a linear combination
of delta-functions, supported at the points  $\Omega_{is}(y)$, with the coefficients
$b_{is}(y)$ for $y\in\Gamma_i$. Then the nonlocal conditions~\eqref{eqNonlocalGeneral2}
and~\eqref{eqNonlocalGeneral1} assume  the form
$$
u|_{\Gamma_i}-\bB_i u=0, \ y\in\Gamma_i,\ i=1,\dots,N;\quad u(y)=0,\ y\in\cK.
$$
\end{remark}

\begin{lemma}\label{l2.3}
Let Conditions  $\ref{cond1.1}$--$\ref{cond1.2}$ hold. Let a function  $u\in C_B(\oG)$
achieve  its positive maximum at a point  $y^0\in\overline G$, and let $Pu\in C(G)$. Then
there is a point  $y^1\in G$ such that  $u(y^1)=u(y^0)$ and $Pu(y^1)\le 0$.
\end{lemma}
\begin{proof}
 If  $y^0\in G$, then the conclusion of the lemma follows from the maximum principle.
Let  $y^0\in\pG$. Assume that the lemma is not true, i.e.,  $u(y^0)>u(y)$ for all $y\in
G$.

Since  $u(y^0)>0$, it follows that $y^0\in \Gamma_i\cap\cO_\varepsilon(\cK)$ for some $i$
and  $b_{is}(y^0)>0$ for some~$s$. Taking into account that $\Omega_{is}(y^0)\in G$ and
$u(y^0)>u(y)$ for all $y\in G$, we have  $u(y^0)-u(\Omega_{is}(y^0))>0$. Therefore,
using~\eqref{eq4-5}, we obtain
$$
0=u(y^0)-\sum\limits_{s=1}^{S_i}b_{is}(y^0)u(\Omega_{is}(y^0))\ge
\sum\limits_{s=1}^{S_i}b_{is}(y^0)(u(y^0)-u(\Omega_{is}(y^0)))>0.
$$
The contradiction proves the lemma.
\end{proof}

\begin{corollary}\label{cor2.1}
Let Conditions $\ref{cond1.1}$--$\ref{cond1.2}$ hold. Let  $u\in C_B(\oG)$ be a solution
of the equation
$$qu(y)-Pu(y)=f(y),\quad y\in G,$$ where  $q>0$ and $f\in
C(\oG)$. Then
\begin{equation}\label{eqcor2.1}
\|u\|_{C(\oG)}\le\dfrac{1}{q}\|f\|_{C(\oG)}.
\end{equation}
\end{corollary}
\begin{proof}
Let $\max\limits_{y\in\oG}|u(y)|=u(y^0)>0$ for some  $y^0\in \oG$. Then, by
Lemma~\ref{l2.3}, there is a point  $y^1\in G$ such that   $u(y^1)=u(y^0)$ and
$Pu(y^1)\le 0$. Hence,
$$
\|u\|_{C(\oG)}=u(y^0)=u(y^1)=\dfrac{1}{q}(Pu(y^1)+f(y^1))\le \dfrac{1}{q}\|f\|_{C(\oG)}.
$$
\end{proof}

\begin{lemma}\label{l2.6}
Let Conditions  $\ref{cond1.1}$--$\ref{cond1.2}$ hold. Then  $\Dom(\bP_B)$ is dense in
$C_B(\oG)$.
\end{lemma}
\begin{proof}
We will follow the scheme proposed in~\cite{GalSkJDE}.

1. Let $u\in C_B(\oG)$. Since $C_B(\oG)\subset C_\cK(\oG)$, it follows that, for any
$\varepsilon>0$ and $q\ge q_1$, there is a function $u_1\in C^\infty(\oG)\cap C_\cK(\oG)$
such that
\begin{equation}\label{eq3.12}
\|u-u_1\|_{C(\oG)}\le\min(\varepsilon,\varepsilon/(2c_1)),
\end{equation}
where $c_1$ is the number from Lemma~\ref{l-th2}.

Set
\begin{equation}\label{eq3.14}
\begin{aligned}
f(y)&\equiv qu_1-P u_1, & & y\in G,\\
\psi_i(y)&\equiv u_1(y)-\bB_iu_1(y),& & y\in\Gamma_i,\ i=1,\dots,N.
\end{aligned}
\end{equation}

Since $u_1\in C_\cK(\oG)$, it follows that $\{\psi_i\}\in \cC_\cK(\pG)$. Using the
relation
$$
u(y)-\bB_iu(y)=0,\qquad
 y\in\Gamma_i,
$$
inequality~\eqref{eq3.12}, and relations~\eqref{eq4-5}, we obtain
\begin{equation}\label{eq3.13}
\|\{\psi_i\}\|_{\cC_\cK(\pG)}\le\|u-u_1\|_{C(\oG)} +\|\{\bB_i
(u-u_1)\}\|_{\cC_\cK(\pG)}\le\varepsilon/c_1.
\end{equation}

Consider the auxiliary nonlocal problem
\begin{equation}\label{eq3.15}
\begin{gathered}
qu_2-P u_2 = f(y),\quad y\in G,\\
u_2(y)-\bB_iu_2(y) =0,\ y\in\Gamma_i;\qquad u_2(y) =0,\ y\in\cK.
\end{gathered}
\end{equation}
Since $f\in C^\infty(\oG)$, it follows from Theorem~\ref{th2} that problem~\eqref{eq3.15}
has a unique solution $u_2\in C_B(\oG)$.

Using~\eqref{eq3.14},~\eqref{eq3.15}, and the relations  $u_1(y)=u_2(y)=0$, $y\in\cK$, we
see that the function   $w_1=u_1-u_2$ satisfies the relations
\begin{equation}\label{eq3.16}
\begin{gathered}
q w_1-P w_1 =0,\quad y\in G,\\
w_1(y)-\bB_iw_1(y)  =\psi_i(y),\ y\in\Gamma_i;\qquad w_1(y) =0,\ y\in\cK.
\end{gathered}
\end{equation}
By Lemma~\ref{l-th2}, problem \eqref{eq3.16} has a unique solution  $w_1\in C_B(\oG)$ and
(taking~\eqref{eq3.13} into account)
\begin{equation}\label{eq3.17}
\|w_1\|_{C(\oG)}\le c_1 \|\{\psi_i\}\|_{\cC_\cK(\pG)}\le c_1 \varepsilon/c_1
=\varepsilon.
\end{equation}

2. Finally, we consider the problem
\begin{equation}\label{eq3.18}
\begin{gathered}
\lambda u_3-Pu_3 =\lambda u_2,\quad y\in G,\\
u_3(y)-\bB_iu_3(y) =0,\ y\in\Gamma_i;\qquad u_3(y) =0,\ y\in\cK.
\end{gathered}
\end{equation}
Since $u_2\in C_B(\oG)$, it follows from Theorem~\ref{th2} that problem~\eqref{eq3.18}
has a unique solution  $u_3\in \Dom(\bP_B)$ for all sufficiently large  $\lambda>0$.

Denote   $w_2=u_2-u_3\in C_B(\oG)$. It follows from~\eqref{eq3.18} that
$$
\lambda w_2-Pw_2 =-P u_2 =f-qu_2.
$$
Applying Corollary~\ref{cor2.1}, we have
$$
\|w_2\|_{C(\oG)}\le\dfrac{1}{\lambda}\|f-qu_2\|_{C(\oG)}.
$$
Choosing sufficiently large $\lambda$ yields
\begin{equation}\label{eq3.20}
\|w_2\|_{C(\oG)}\le\varepsilon.
\end{equation}

Inequalities~\eqref{eq3.12},~\eqref{eq3.17}, and~\eqref{eq3.20} imply
$$
\|u-u_3\|_{C(\oG)}\le\|u-u_1\|_{C(\oG)}+\|u_1-u_2\|_{C(\oG)}+ \|u_2-u_3\|_{C(\oG)}\le
3\varepsilon.
$$
\end{proof}

Now we can prove the main result about the existence of a Feller semigroup.
\begin{theorem}\label{th2.1}
Let conditions  $\ref{cond1.1}$--$\ref{cond1.2}$ hold. Then the operator
$\bP_B:\Dom(\bP_B)\subset C_B(\oG)\to C_B(\oG)$ is a generator of a Feller semigroup.
\end{theorem}
\begin{proof}
1. By Lemma~\ref{l2.6}, the domain of the operator  $\bP_B$ is dense in $C_B(\oG)$.

2. By Theorem~\ref{th2} and Corollary~\ref{cor2.1}, there exists a bounded operator
$(qI-\bP_B)^{-1}: C_B(\oG)\to C_B(\oG)$ for all sufficiently large  $q>0$ and
$$
\|(qI-\bP_B)^{-1}\|\le 1/q.
$$

3. Let us prove that the operator  $(qI-\bP_B)^{-1}$ is nonnegative. Assume the contrary.
Then there is a function  $f\ge0$ such that the solution $u\in\Dom(\bP_B)$ of the
equation $qu-\bP_Bu=f$ achieves its negative minimum at some point  $y^0\in \oG$.
Therefore, the function  $v=-u$ achieves its positive maximum at the point  $y^0$. Due to
Lemma~\ref{l2.3}, there exists a point  $y^1\in G$ such that $v(y^1)=v(y^0)$ and $\bP_B
v(y^1)\le 0$. Hence,   $ 0<v(y^0)=v(y^1)=(\bP_Bv(y^1)-f(y^1))/q\le 0. $ This
contradiction shows that  $u\ge0$.

Thus all the hypotheses of the Hille--Iosida theorem (Theorem~\ref{thHI}) hold, and the
operator  $\bP_B:\Dom(\bP_B)\subset C_B(\oG)\to C_B(\oG)$ is a generator of a Feller
semigroup.
\end{proof}

\smallskip

The author is grateful to Prof. A.L. Skubachevskii for attention
to this work.

\smallskip

The work was supported by Russian Foundation for Basic Research (project No.~07-01-00268)
and the Alexander von Humboldt Foundation.


\begin{thebibliography}{99}

\bibitem{Ventsel}
A. D. Ventsel, ``On boundary conditions for multidimensional diffusion processes,'' {\it
Teor. Veroyatnost. i Primen., \bf 4}, 172--185 (1959); English transl.: {\it Theory
Probab. Appl.}, {\bf 4} (1959).

\bibitem{GalSkMs}
E. I. Galakhov and A. L. Skubachevskii, ``On contractive nonnegative semigroups with
nonlocal conditions,'' {\it Mat. Sb.}, {\bf 189}, 45--78 (1998); English transl.: {\it
Math. Sb.} {\bf 189} (1998).

\bibitem{GurPetr03}
P.~L.~Gurevich, ``Asymptotics of solutions for nonlocal elliptic problems in plane
angles,'' {\it Trudy seminara imeni I.~G.~Petrovskogo, \bf 23}, 93--126 (2003); English
transl.: {\it J. Math. Sci.,} {\bf 120}, No.~3, 1295--1312 (2004).

\bibitem{GurSkub}
P. L. Gurevich and A. L. Skubachevskii, ``On the fredholm and unique solvability  of
nonlocal elliptic problems in multidimensional domains,'' {\it Trudy Moskov. Mat.
Obshch.}, {\bf 268}, 288--373 (2007); English transl.: {\it Trans. Moscow Math. Soc.}



\bibitem{NP}
S. A. Nazarov and  B. A. Plamenevskii, {\it Elliptic Problems in Domains with Piecewise
Smooth Boundaries,}   De Gruyter Expositions in Mathematics, {\bf 13}. Walter de Gruyter
Publichers, Berlin --- New York, 1994.

\bibitem{SkDAN89}
A. L. Skubachevskii, ``On some problems for multidimensional diffusion
     processes,'' {\it Dokl. Akad. Nauk SSSR}, {\bf 307}, 287--291 (1989); English
     transl.: {\it Soviet Math. Dokl.\/} {\bf 40} (1990).

\bibitem{SkDu90}
A.~L.~Skubachevskii, ``Model nonlocal problems for elliptic equations in dihedral
angles,'' {\it Differentsial'nye Uravneniya, \bf 26}, 119--131 (1990); English transl.:
{\it Differ. Equ., \bf 26} (1990).



\bibitem{BCP}
J.~M. Bony, P. Courrege, and P. Priouret,   ``Semi-groups de Feller sur une vari\'et\'e
\`a bord compacte et probl\`emes aux limites int\'egro-diff\'erentiels du second ordre
donnant lieu au principe du maximum,''  {\it Ann. Inst. Fourier (Grenoble)}, {\bf 18},
369--521 (1968).

\bibitem{Feller1}
W. Feller,  ``The parabolic differential equations and the associated semi-groups of
transformations,''   {\it Ann.   Math.\/}, {\bf 55}, 468--519 (1952).

\bibitem{Feller2}
W.  Feller,   ``Diffusion processes in one dimension,'' {\it Trans. Amer. Math. Soc.},
{\bf 77}, 1--30 (1954).

\bibitem{GalSkJDE}
E. I. Galakhov and A. L. Skubachevskii, ``On Feller semigroups generated by elliptic
operators with integro-differential boundary conditions,'' {\it J. Differential
Equations},
  {\bf 176},   315--355 (2001).


\bibitem{GurGiess}
P.~L. Gurevich,  ``Nonlocal problems for elliptic equations in dihedral angles and the
Green formula,'' {\it In} {\it Mitteilungen aus dem Mathem. Seminar Giessen, Math. Inst.
Univ. Giessen, Germany},   {\bf  247},  1--74 (2001).







\bibitem{GurRJMP04}
P.~L. Gurevich, ``Solvability of nonlocal elliptic problems in Sobolev spaces,~II,'' {\it
Russ. J. Math. Phys.},     {\bf 11}, No.~1, 1--44 (2004).

\bibitem{Ishikawa}
Y.  Ishikawa,    ``A remark on the existence of a diffusion process with non-local
boundary conditions'' {\it J. Math. Soc. Japan},    {\bf 42}, 171--184  (1990).



\bibitem{SU}
K. Sato  and T. Ueno, ``Multi-dimensional diffusion and the Markov process on the
boundary,'' {\it J. Math. Kyoto Univ.},   {\bf  4}, 529--605 (1965).


\bibitem{SkRJMP95}
A. L.  Skubachevskii,  ``Nonlocal elliptic problems and multidimensional diffusion
processes'' {\it Russian J.  Mathematical Physics},   {\bf 3}, 327--360 (1995).



\bibitem{Taira1}
K.  Taira,  {\it Diffusion Processes and Partial Differential Equations.\/} New York
--- London: Academic Press, 1988.

\bibitem{Taira3}
K. Taira, {\it Semigroups, boundary value problems and Markov processes.\/}   Berlin:
Springer-Verlag, 2004.

\bibitem{Watanabe}
S. Watanabe,   ``Construction of diffusion processes with Wentzell's boundary conditions
by means of Poisson point processes of Brownian excursions,'' In: {\it Probability
Theory}, {\bf 5},  255--271 (1979).  Banach Center Publications. Warsaw: Polish
Scientific Publishers, 1979.
\end{thebibliography}
\end{document}